\documentclass[10pt]{article}
\usepackage{amssymb}
\usepackage{amsthm}
\usepackage{amsmath}
\usepackage{graphicx}
\usepackage{mathrsfs}
\usepackage{amsfonts}
 \usepackage[marginal]{footmisc}
\newtheorem{defn}{Definition}[section]
\newtheorem{thm}{Theorem}[section]
\newtheorem{lemma}{Lemma}[section]
\newtheorem{rem}{Remark}[section]

 \numberwithin{equation}{section}

 \newcounter{Romannumber}
\newcommand{\MyRoamn}[1]
{\setcounter{Romannumber}
{#1}\Roman{Romannumber}}

\allowdisplaybreaks

\def\f{\frac}
\def\rr{\mathbb R}
\def\pa{\partial}

\def\biaoti#1{
                \begin{center}
                    {\bf\LARGE #1}
                \end{center}\bigskip\bigskip
}

\begin{document}

\biaoti{Exact boundary controllability and exact boundary   synchronization for a coupled system of wave equations    with coupled Robin boundary controls }

 \vskip 0.5cm

\centerline{\bf \large Tatsien LI\footnote{ School of Mathematical Sciences, Fudan
University, 200433 Shanghai, China;  Shanghai Key Laboratory for
Contemporary Applied Mathematic;  Nonlinear Mathematical Modeling
and Methods Laboratory, dqli@fudan.edu.cn. Projet supported by the National Natural Science Foundation of China (No 11831011).}, \quad Xing LU\footnote {Corresponding Author. School of Mathematics, Southeast University, 211189  Nanjing, China, xinglu@seu.edu.cn. Projet supported by the National Natural Science Foundation of China (No 11901082), the Natural Science Foundation of Jiangsu Province (BK20190323), and the Fundamental Research Funds for the Central Universities.} \quad  Bopeng RAO\footnote {Institut de Recherche
 Math\'ematique Avanc\'ee, Universit\'e de Strasbourg, 67084
 Strasbourg, France,  bopeng.rao@math.unistra.fr.}}
 \vskip 0.4cm

\begin{abstract} 
In this paper, we consider the exact boundary controllability and the exact boundary synchronization (by groups) for a coupled system of wave equations     with coupled Robin boundary controls. Owing to the difficulty coming from the lack of regularity of the solution, we confront a bigger challenge than that in the case with Dirichlet or Neumann boundary controls. In order to overcome this difficulty, we use the regularity results of solutions to the mixed  problem with Neumann boundary conditions by Lasiecka and Triggiani (\cite{Lasiecka}) to get the regularity of   solutions to the mixed problem with coupled Robin  boundary conditions. Thus we show the exact boundary controllability of the system, and by a method of compact perturbation, we obtain  the non-exact boundary controllability of the system with fewer boundary controls  on some special domains.  Based on this, we further  study the exact boundary synchronization (by groups)   for the same  system,  the determination of the exactly synchronizable state  (by groups), as well as the necessity of the compatibility conditions of the coupling matrices.

\end{abstract}

\indent {\bf Keywords } \quad
Exact boundary controllability, exact boundary  synchronization,   coupled system of wave equations, coupled Robin boundary controls.

\indent {\bf 2000 MR Subject Classification } 
  93B05, 93B07, 93C20

\section{Introduction}\label{sec1}

Synchronization is a widespread natural phenomenon. It was first observed by Huygens in 1665 (\cite{Huygens}). The theoretical research on synchronization from the  mathematical point of view dates back to N. Wiener in 1950s (see Chapter 10 in \cite{wiener}, pp.199). Since 2012,  Li and Rao started the research on the synchronization for coupled systems governed by PDEs, and they showed that the synchronization in this case could be realized  in a finite time by  means of proper boundary controls. Consequently, the study of synchronization becomes a part of research in control theory. 
Precisely speaking, Li and Rao considered the exact boundary synchronization for a coupled system of wave equations with Dirichlet boundary controls for any given space dimensions in the framework of weak solutions (\cite{Raofr, Rao3, LiRao2}) and for the one-space-dimensional case 
in the framework of classical solutions (\cite{hulirao, Hu, Lu0}). Corresponding results were expanded to the exact boundary synchronization by $p(\geq 1)$ groups (\cite{Rao2, Rao6}). Moreover, 
Li and Rao proposed  the concept of approximate boundary null controllability and approximate boundary synchronization in \cite{Li4} and \cite{Raokalman} and further studied them.

Throughout this paper,  $\Omega\subset  \mathbb{R}^n $ is  a bounded domain with smooth boundary $\Gamma $ or a parallelepiped.
In the first situation, we  assume that $\Omega$ satisfies the usual   multiplier geometrical  condition (\cite{Lions3}). Without loss of generality, assume that there exists an $x_0\in \mathbb{R}^n$, such that by setting $m=x-x_0$,  we have
\begin{align}\label{geo}
(m, \nu)>0, \quad  \forall x\in \Gamma, \end{align}
where $\nu$ is the unit outward normal vector on the boundary, and $(\cdot, \cdot)$ denotes the inner product in $\mathbb{R}^n $.

We define
\begin{align} \label{hdef} {\mathcal H}_0=\Big \{u: ~u \in L^2   (\Omega), \int_{\Omega}udx=0 \Big\}, \quad {\mathcal H}_1 =  H^1   (\Omega)\cap{\mathcal H}_0. \end{align}

Inspired by the synchronization of the system with Dirichlet boundary controls, Li, Lu and Rao studied the null controllability and synchronization for the
following coupled system of wave equations with $\text{Neumann}$ boundary controls on a bounded domain $\Omega\subset  \mathbb{R}^n $  with smooth boundary:
\begin{align}\label{Bn}\begin{cases}
U''-{\Delta} U+AU=0  & \hbox{in} \quad (0,+\infty) \times \Omega,\\
\partial_\nu U = DH  & \hbox{on} \quad  (0,+\infty) \times \Gamma,  \end{cases}
\end{align}
where,
 the coupling matrix
$A=(a_{ij})$ is of order $N$, the boundary control matrix $D$ is an $N\times M ~(M\leq N)$ full column-rank matrix, namely, rank$(D)=M$, and both $A$  and $D$ have real constant elements, $U = (u^{(1)},\cdots, u^{(N)} )^T$ and $H= (h^{(1)} ,\cdots, h^{(M)} )^T$ denote the state variables  and the boundary controls, respectively. The discussion on the control problem will become more flexible because of the introduction of the boundary control matrix $D$. Moreover,   $\partial_\nu$ denotes the outward normal derivative on the boundary.

\begin{rem}
Corresponding results on the exact boundary synchronization and the approximate boundary synchronization  obtained in \cite{lu2}, \cite{lu3} and \cite{Raoneu} were originally presented for the following system
\begin{align}
\begin{cases}
U''-{\Delta} U+AU=0  & \hbox{in} \quad (0,+\infty) \times \Omega,\\
U= 0  & \hbox{on} \quad  (0,+\infty) \times \Gamma_0,\cr
\partial_\nu U = DH  & \hbox{on} \quad  (0,+\infty) \times \Gamma_1  \end{cases}
\end{align}
on a bounded domain $\Omega\subset  \mathbb{R}^n $ with smooth boundary   $\Gamma = \Gamma_1\cup\Gamma_0$ with $\overline\Gamma_1\cap\overline\Gamma_0=\emptyset $ and mes$(\Gamma_1) \not = 0$, where  mes$(\cdot)$ stands for the Lebesgue's surface measure on  $\Gamma$. However, using the basic spaces defined by \eqref{hdef}, all those results can be obtained for system \eqref{Bn}.
\end{rem}


 We have

\begin{lemma}\label{controllableneu}
Assume that $\Omega\subset  \mathbb{R}^n $ is  a smooth bounded domain. 
Assume furthermore that  $M=\text{rank}(D)=N$. Then there exists a $T>0$, for any given initial data $(\widehat{U}_0, \widehat{U}_1)\in ({\cal H}_1)^N\times  ({\cal H}_0)^N$, there exists a boundary control  $H \in L^2(0, T;
(L^2(\Gamma_1))^N)$, 
such that the corresponding solution $U=U(t,x)$ to system \eqref{Bn} satisfies
\begin{align}
t \geq T: \quad U(t,x) \equiv 0, ~~x \in \Omega,
\end{align}
namely, system \eqref{Bn} is exactly null controllable at the time $T$. \end{lemma}

\begin{rem}\label{1.1}
By  the method given in \cite{Raoneu}, the boundary control    $H$ can be chosen to continuously depend on the initial data:
\begin{align}
\|H\|_{L^2(0, T, (L^2(\Gamma_1))^{N})}\leq c\| (\widehat{U}_0, \widehat{U}_1)\|_{({\cal H}_1)^N \times ( {\cal H}_0)^N},\end{align}
here and hereafter, $c$ is a positive constant independent of the initial data.
\end{rem}

On the other hand,  when there is a lack of boundary controls, we have
\begin{lemma}\label{uncontrollableneu}
Assume that $\Omega\subset  \mathbb{R}^n $ is  a smooth bounded domain. 
When $M=\text{rank}(D)<N$, no matter how large $T>0$ is, system \eqref{Bn} is not exactly controllable at the time $T$ in the space $({\cal H}_1)^N \times  ({\cal H}_0)^N$. 
\end{lemma}


The study on the synchronization will be more difficult with  more complicated boundary conditions. In this paper we will consider  a  coupled system of wave equations with coupled $\text{Robin}$ boundary controls as follows:
\begin{align}\label{B}\begin{cases} U''-{\Delta} U+AU=0 &  \hbox{in} \quad   (0,+\infty) \times \Omega,\\
\partial_\nu U + BU = DH  & \hbox{on} \quad   (0,+\infty) \times \Gamma  \end{cases}
\end{align}
with the corresponding initial condition
\begin{align}\label{Bin}
t=0:\quad U=\widehat{U}_0,\quad U'=\widehat{U}_1 \qquad  \hbox{in} \quad  \Omega,
\end{align}where $B=(b_{ij})$ is the boundary coupling matrix of order $N$ with constant elements.

To study the exact boundary controllability and the exact boundary synchronization for a coupled system of wave equations with coupled $\text{Robin}$ boundary controls,  most of  difficulties come from the complicated form of boundary conditions. To deeply study the non-exact boundary controllability and the necessity of the conditions of $C_p$-compatibility for the exact synchronization by $p$-groups,
we have to further study the regularity  of solutions to Robin problem, then these problems  can be obtained on some special domains by a method of compact perturbation, based on the improved regularity results.

\section{Regularity of solutions with Neumann boundary conditions}\label{sec2}
Similarly to the  problem of wave equations with Neumann boundary conditions, a problem with Robin boundary conditions  no longer enjoys the hidden regularity as in the case  with  Dirichlet boundary conditions. As a result, 
the solution to  problem \eqref{B}--\eqref{Bin} with Robin boundary conditions  is not smooth enough in general  for the proof of the non-exact boundary controllability of the system. In order to overcome this difficulty, we should  deeply study the regularity of solutions to wave equations with Neumann boundary conditions. For this purpose, we will review some existing results with Neumann boundary conditions.

 Consider the following second order hyperbolic problem
on a  bounded domain $\Omega \subset \mathbb{R}^n$ with  boundary $\Gamma$: \begin{align}\label{las}\begin{cases} y_{tt}+A(x, \pa)y = f   & \hbox{in} \quad (0,T) \times \Omega,\\
\f{\pa y}{\partial \nu_A}  = g  & \text{on} \quad  (0,T) \times \Gamma, \\
t=0:~~y=y_0,\quad y_t=y_1 & \text{in} \quad   \Omega,  \end{cases}\end{align}
where
\begin{align}\label{oper}
A(x, \pa)=-\sum_{i,j=1}^n a_{ij}(x)\f{\pa^2}{\pa x_i \pa x_j} +\sum_{i=1}^n b_i(x)\f{\pa}{\pa x_i} +c_0(x),
\end{align} 
in which $a_{ij}(x)$  with $a_{ij}(x)=a_{ji}(x)$, $b_i(x)$ and $c_0(x)$ are  smooth real coefficients, and the principal part of $A(x, \pa)$ is supposed to be uniformly strong elliptic in $\Omega$:
\begin{align}
\sum_{i,j=1}^n a_{ij}(x)\eta_i\eta_j \geq c \sum_{i=1}^n \eta_i^2   \end{align}
for any given $x \in \Omega$ and for any given $\eta=(\eta_1, \cdots, \eta_n )\in \rr^n$, where $ c>0$ is a positive constant; moreover, $\f{\pa y}{\partial \nu_A}$ is the outward normal derivative associated  with $A$:
\begin{align}
\f{\pa y}{\partial \nu_A}=   \sum_{i=1}^N \sum_{j=1}^N a_{ij}(x) \f{\pa{y}}{\pa x_i} \nu_j
,\end{align}
  $\nu=(\nu_1,\cdots, \nu_n)^T$ being the unit outward normal vector  on the boundary $\Gamma$.

Define the operator $\cal A$  by
\begin{align}\label{operator}{\cal A} =A(x, \pa),\quad  {\cal D}({\cal A})=\big\{ y\in H^2(\Omega): \f{\pa y}{\partial \nu_A}=0~\text{on}~\Gamma  \big\}. \end{align}

%
%

In \cite{Lasiecka},  Lasiecka and Triggiani got the optimal regularity for the solution to   problem \eqref{las} 
by means of the theory of cosine operator for $\Omega \subset \mathbb{R}^n (n \geq 2)$. 
On the other hand, when $n=1$, better results can be obtained (see Theorems 3.1--3.3 and Remarks 3.1 in \cite{Lasiecka3}). Moreover, more regularity  results  can be proved when  the domain is  a parallelepiped. For conciseness and clarity, we list only those results which are needed in this paper. 

Let $\epsilon >0$ be an arbitrarily given small number. Here and hereafter, we always assume that $\alpha,\beta $ are given, respectively,  as follows: 
\begin{align}\label{alfarobin}
\begin{cases}\alpha=3/5-\epsilon,~\beta=3/5, & \text{$\Omega$ is  a   smooth  bounded domain}   \\& \text{and}~A(x, \pa)~\text{is defined by \eqref{oper}}; \\
\alpha=\beta=3/4-\epsilon, & \text{$\Omega$ is a parallelepiped  and 
$A(x, \pa)=-\Delta$}.
\end{cases} \end{align}

%
\begin{lemma}
\label{Smoothneu}
Let $\Omega\subset  \mathbb{R}^n $ be a bounded domain with   boundary $\Gamma$. 
Assume that $y_0 \equiv y_1 \equiv 0$ and  $f \equiv 0$. 
For any given $g\in L^2(0, T;L^2(\Gamma))$,  the unique solution $y$ to   problem \eqref{las}
 satisfies
\begin{equation}(y, y')\in C^0([0, T];  H^{\alpha}(\Omega)\times  H^{\alpha-1}(\Omega)) \label{c31}\end{equation}
and
\begin{equation}y|_{\Sigma}\in H^{2\alpha-1}(\Sigma)= L^2(0,T;H^{2\alpha-1}(\Gamma))\cap H^{2\alpha-1}(0,T;L^2(\Gamma)),\label{c41}\end{equation}
where $H^{\alpha}(\Omega)$ denotes the usual Sobolev space of order $\alpha$ and $\Sigma=(0,T) \times \Gamma$.
\end{lemma}

\begin{lemma} 
\label{Smoothneu2}
Under the assumption on $\Omega$ given in Lemma \ref{Smoothneu}, assume that $y_0 \equiv y_1 \equiv 0$ and $g \equiv 0$. 
For any given $f \in L^2(0, T;L^2(\Omega))$, the unique solution $y$ to   problem \eqref{las} 
 satisfies
\begin{equation}(y, y')\in C^0([0, T];  H^{1}(\Omega)\times  L^2(\Omega))\label{c311}\end{equation}
and
\begin{equation}y|_{\Sigma}\in H^{\beta}(\Sigma).\label{c411}\end{equation}
\end{lemma}

\begin{lemma}
\label{Smoothneu3}
Under the assumption on $\Omega$ given in Lemma \ref{Smoothneu}, assume that $f \equiv 0$ and $g \equiv 0 $. 

$(1)$~If $(y_0, y_1) \in H^1(\Omega)\times L^2(\Omega) $, then the unique solution $y$ to   problem \eqref{las}  satisfies
\begin{equation}\label{2.50}(y, y')\in C^0([0, T];  H^{1}(\Omega)\times  L^2(\Omega)) \end{equation}
and
\begin{equation}\label{2.51}y|_{\Sigma}\in H^{\beta}(\Sigma).\end{equation}

$(2)$ If $(y_0, y_1) \in L^2(\Omega) \times (H^1(\Omega))' $, where $ (H^1(\Omega))' $ denotes the dual space of $H^1(\Omega) $ with respect to $L^2(\Omega)$, then  the unique solution $y$ to   problem \eqref{las}  satisfies
\begin{equation}\label{2.57}(y, y')\in C^0([0, T]; L^2(\Omega) \times (H^1(\Omega))' ) \end{equation}
and
\begin{equation}\label{2.58} y|_{\Sigma}\in H^{\alpha-1}(\Sigma). \end{equation}
\end{lemma}



\begin{rem}
In the  results mentioned above,
the  mappings from the given data to the solution are all continuous with respect to the corresponding topologies.
\end{rem}

\section{Well-posedness of a coupled system of wave equations  with coupled Robin boundary conditions}

Let $\Omega \subset \mathbb{R}^n $ be a smooth bounded domain or a  parallelepiped as mentioned before. We now prove the well-posedness of problem  \eqref{B}  and  \eqref{Bin}.

Let $\Phi= (\phi^{(1)},\cdots, \phi^{(N)})^T.$ We first consider the following  adjoint system
\begin{align} \label{dual}\begin{cases}
\Phi''-{\Delta} \Phi +A^T\Phi=0\hskip2.1cm~& \hbox{in} \quad    (0,+\infty) \times \Omega,\\
\partial_\nu \Phi + B^T\Phi = 0 \hskip3cm~& \hbox{on} \quad  (0,+\infty) \times \Gamma  \\
\end{cases}
\end{align}
with the initial data
\begin{equation}\label{3.2new} t=0:\quad \Phi=\widehat{\Phi}_0, \quad  \Phi'=\widehat{\Phi}_1\quad  \hbox{in}  ~~ \Omega,\end{equation}
where $A^T$ and $B^T$ denote the transpose of   $A $ and $B $, respectively.

\begin{thm}\label{welldual}
Assume that  $\Omega \subset \mathbb{R}^n $ is a smooth bounded domain or a  parallelepiped. Assume furthermore that $B$ is similar to a real symmetric matrix. Then for any given $(\widehat{\Phi}_0,\widehat{\Phi}_1) \in (\mathcal H_1)^N \times (\mathcal H_0)^N$, the adjoint problem \eqref{dual}--\eqref{3.2new} admits a unique weak solution 
\begin{equation}\label{re}(\Phi, \Phi')\in  C^0_{loc}([0,+\infty); (\mathcal  H_1)^N \times (\mathcal H_0 )^N)  \end{equation}
in the sense of  $C_0$-semigroup, where  $\mathcal H_1$ and $\mathcal H_0$ are defined by \eqref{hdef}.
\end{thm}
\begin{proof}
Without loss of generality, we assume that $B$ is   a   real symmetric matrix.

 We  formulate  system \eqref{dual} into the following variational form:
\begin{equation}\int_\Omega (\Phi'',  \widehat{\Phi})dx + \int_\Omega \langle \nabla\Phi, \nabla  \widehat{\Phi} \rangle dx+
\int_{\Gamma}(\Phi, B \widehat{\Phi})d\Gamma +\int_\Omega (\Phi,   A\widehat{\Phi})dx =0 \label{3.4new}\end{equation}
for any  given test function $ \widehat{\Phi} \in (\mathcal H_1)^N$, where $(\cdot, \cdot)$ denotes the inner product of  $\mathbb R^N$, while
$\langle\cdot, \cdot\rangle$  denotes  the  inner product of $\mathbb M^{N\times N}(\mathbb R)$.

Recalling  the following  interpolation inequality (\cite{Liu})
$$\int_{\Gamma}|\phi|^2d\Gamma\leqslant  c \|\phi\|_{H^1(\Omega)}\|\phi\|_{L^2(\Omega)},\quad \forall \phi\in H^1(\Omega),$$
 we have
$$\int_{\Gamma}(\Phi, B\Phi)d\Gamma \leqslant \|B\|\int_{\Gamma}|\Phi|^2d\Gamma
\leqslant c\|B\|\|\Phi\|_{(\mathcal H_1)^N}\|\Phi\|_{(\mathcal H_0)^N},$$
then it is easy to see  that
$$\int_\Omega \langle \nabla\Phi, \nabla\Phi\rangle dx+
\int_{\Gamma}(\Phi,B\Phi)d\Gamma +\lambda \|\Phi\|_{(\mathcal H_0)^N}^2\geqslant c'\|\Phi\|_{(\mathcal H_1)^N}^2$$
for some suitable constants $\lambda>0$ and $c'>0$.
Moreover, the non-symmetric part in \eqref{3.4new} satisfies
$$\int_\Omega (\Phi,   A \widehat{\Phi})dx\leqslant \|A\|\|\Phi\|_{(\mathcal H_0)^N}\|  \widehat{\Phi}\|_{(\mathcal H_0)^N}.$$
By Theorem 1.1 of Chapter 8 in \cite{Lions1961}(p.151), the variational problem (\ref{3.4new})  with the initial data \eqref{3.2new} admits a unique solution $\Phi$ with the smoothness \eqref{re}. The proof is complete.
\end{proof}

\begin{defn}
$U$ is a weak solution to the mixed problem \eqref{B}--\eqref{Bin}, if
\begin{equation} \label{weak}U \in  C_{loc}^0([0, +\infty);  (\mathcal H_0)^N)\cap C_{loc}^1([0, +\infty);  (\mathcal H_{-1})^N),\end{equation}
where $\mathcal H_{-1}$ denotes the dual space of  $\mathcal H_{1}$ with the pivot space $\mathcal H_0$, such that for  any given $(\widehat{\Phi}_0,\widehat{\Phi}_1) \in   ( \mathcal  H_1)^N \times (\mathcal H_0)^N$ and for all given $t\geqslant  0$, we have
\begin{align}\label{defi}
\langle\!\langle &(U'(t), -U (t)),(\Phi(t), \Phi'(t))\rangle \!\rangle  \notag \\
& =\langle\!\langle (\widehat{U}_1,-\widehat{U}_0),( \widehat{\Phi}_0, \widehat{\Phi}_1)\rangle \!\rangle
+\int_0^t \int_{\Gamma}(DH(\tau), \Phi(\tau)) dxdt,
\end{align} in which $\Phi(t)$ is the solution to the adjoint problem \eqref{dual}--\eqref{3.2new}, 
and 
 $\langle\!\langle\cdot,   \cdot \rangle \!\rangle$ denotes the duality between the spaces $({\mathcal H}_{-1})^N \times ({\mathcal H}_0)^N$ and
 $({\mathcal H}_1)^N \times ({\mathcal H}_0)^N$.
\end{defn}

\begin{thm}\label{exist}
Let $\Omega \subset \mathbb{R}^n $ be a smooth bounded domain or a  parallelepiped.  Assume that $B$ is similar to a real symmetric matrix. For any given $H\in L_{loc}^2(0, +\infty; (L^2(\Gamma))^M)$ and $(\widehat{U}_0,\widehat{U}_1) \in  (\mathcal  H_0 )^N \times (\mathcal H_{-1})^N$, 
problem \eqref{B}--\eqref{Bin} admits a unique weak solution $U$. 
Moreover, the mapping
 $$(\widehat{U}_0,\widehat{U}_1, H) \rightarrow (U,U') $$
is continuous with respect to the corresponding topologies.\end{thm}

\begin{proof}
Let $\Phi$ be the solution to the adjoint problem  \eqref{dual}--\eqref{3.2new}.

Define a linear functional as follows:
\begin{align}
L_t(\widehat{\Phi}_0,\widehat{\Phi}_1)= \langle\!\langle (\widehat{U}_1,-\widehat{U}_0),( \widehat{\Phi}_0, \widehat{\Phi}_1)\rangle\!\rangle
+\int_0^t \int_{\Gamma}(DH(\tau), \Phi(\tau)) dxdt.
\end{align}Clearly,   $L_t$ is bounded in $(\mathcal   H_1)^N \times (\mathcal   H_0)^N$. Let $S_t$ be the semigroup in $(\mathcal   H_1)^N \times (\mathcal   H_0)^N$, corresponding to the adjoint problem \eqref{dual}--\eqref{3.2new}.   $L_t\circ S_t^{-1}$ is bounded in $(\mathcal   H_1)^N \times (\mathcal   H_0)^N$. Then, by Riesz-Fr\'{e}chet representation theorem, for any given $(\widehat{\Phi}_0,\widehat{\Phi}_1) \in  (\mathcal   H_1)^N \times (\mathcal   H_0)^N$, there exists a unique $(U'(t),-U(t))\in (\mathcal  H_{-1})^N \times (\mathcal   H_0)^N$, such that
\begin{align}
L_t\circ S_t^{-1}(\Phi(t),\Phi'(t))= \langle\!\langle (U'(t),-U(t)),(\Phi(t),\Phi'(t))\rangle\!\rangle.
\end{align}
By
\begin{align}
L_t\circ S_t^{-1}(\Phi(t),\Phi'(t))= L_t(\widehat{\Phi}_0,\widehat{\Phi}_1)
\end{align}for any given $(\widehat{\Phi}_0,\widehat{\Phi}_1) \in  ( \mathcal   H_1)^N \times (\mathcal   H_0)^N$, \eqref{defi} holds, then $(U,U')$ is the unique weak solution to problem  \eqref{B}--\eqref{Bin}. Moreover, we have
\begin{align}
\|(U'(t),-U(t))\|_{ (\mathcal   H_{-1})^N \times (\mathcal   H_0)^N}=\|L_t\circ S_t^{-1}\| \notag \qquad\qquad\qquad\qquad \\
\leq c(\|(\widehat{U}_0, \widehat{U}_1)\|_{ (\mathcal   H_0)^N\times( \mathcal   H_{-1})^N }+\|H\|_{L^2(0,T; (L^2(\Gamma))^M)})
\end{align}for all $t \in [0,T]$.

At last, by a classic argument of density, we obtain the regularity desired by \eqref{weak}.\end{proof}

\begin{rem}
From now on, in order to guarantee the well-posednessz of  problem  \eqref{B}--\eqref{Bin},  we always assume that $B$ is similar to a real symmetric matrix. This condition  is also required for the well-posedness of weak solution even in one-space-dimensional case.
However, the exact boundary controllability and the exact boundary synchronization of classical solutions in one-space-dimensional case were  done   without the symmetry of $B$  in \cite{Li} and \cite{Hu}.
\end{rem}

\section{Regularity of solutions 
with coupled Robin boundary conditions}

In this section, we will improve the regularity results for Robin problem  by means of the regularity results for Neumann problem mentioned in Section \ref{sec2}. 

\begin{thm}  \label{Smooth} 
When $ \Omega \subset \mathbb{R}^n$ is a  smooth  bounded domain or a parallelepiped,  for any given $H\in L^2(0, T; (L^2(\Gamma))^M)$ and any given $ (\widehat{U}_0,\widehat{U}_1) \in (\mathcal   H_1)^N \times (\mathcal H_0)^N$, the weak solution $U$ to  
problem \eqref{B}--\eqref{Bin} satisfies
\begin{equation}(U, U')\in C^0([0, T];   (H^{\alpha}(\Omega))^N \times (H^{\alpha-1}(\Omega))^N)\label{sm1}\end{equation}
and
\begin{equation}U|_{\Sigma}\in (H^{2\alpha-1}(\Sigma))^N,\label{sm2}\end{equation}
where $\Sigma=(0, T)\times \Gamma$, and $\alpha$ is defined by \eqref{alfarobin}. Moreover, the linear mapping
$$(\widehat{U}_0,\widehat{U}_1, H) \rightarrow (U, U') $$
 is continuous with respect to the corresponding topologies.
\end{thm}

\begin{proof} We first consider the case that  $\Omega$ is sufficiently smooth, for example, with $C^3$ boundary. There exists a function $h\in C^2(\overline\Omega)$, such that
\begin{equation}\nabla h=\nu\quad \hbox{on}\quad  \Gamma,\end{equation}
where $\nu$ is the unit outward normal vector on the boundary $\Gamma$(\cite{Lions3}).

Noting \eqref{2.50} and \eqref{2.51} in Lemma \ref{Smoothneu3}, it is easy to see that we need only to consider the case  $\widehat{U}_0 \equiv \widehat{U}_1 \equiv 0$.

Let $\lambda$ be an eigenvalue of  $B^T$ and let $e$ be the corresponding   eigenvector:
$$ B^T e =\lambda e.$$
Defining
\begin{equation}\phi = (e, U),\end{equation}
we have
\begin{equation}\label{4.5}\left\lbrace
\begin{array}{ll}\phi''-\Delta\phi =-(e, AU)  & \hbox{in} \quad (0,T) \times \Omega,\\
\partial_\nu \phi + \lambda\phi = (e, DH)   & \hbox{on} \quad (0,T) \times \Gamma,\\
t=0:~\phi=0, ~ \phi'=0  & \hbox{in} \quad \Omega.
\end{array}
\right.
\end{equation}(Equation \eqref{4.5} is actually valid in the weak sense,   for simplicity of presentation, however, we  write it in the classical sense here and hereafter.)
Let
\begin{equation}\psi = e^{\lambda h}\phi.\end{equation}
Problem \eqref{B}--\eqref{Bin} can be rewritten into the following problem with Neumann boundary conditions:
\begin{equation}\left\lbrace
\begin{array}{lll}\psi''-\Delta\psi +b(\psi)=-e^{\lambda h}(e, AU)\quad &\hbox{in}
\quad  (0,T) \times \Omega,\\
\partial_\nu \psi = e^{\lambda h}(e, DH) & \hbox{on} \quad  (0,T) \times \Gamma,\\
t=0:\quad \psi=0,\quad \psi'=0 & \hbox{in} \quad   \Omega,\end{array}
\right.
\end{equation}
where $b(\psi) = 2\lambda\nabla h\cdot\nabla\psi + \lambda   (\Delta h-\lambda
|\nabla h|^2)\psi$ is a  first order  linear form of  $\psi$ with smooth coefficients.

By Theorem \ref{exist}, $U\in C^0([0, T]; ( \mathcal H_0)^N )$. By \eqref{c311} in Lemma \ref{Smoothneu2},
the solution $\psi$ to the following problem with homogeneous Neumann boundary conditions:
\begin{equation}\left\lbrace
\begin{array}{ll}
\psi''-\Delta\psi +b(\psi)=-e^{\lambda h}(e, AU)\quad & \hbox{in}
\quad  (0,T) \times\Omega,\\
\partial_\nu \psi =0 & \hbox{on} \quad  (0,T) \times\Gamma,\\
t=0:\quad  \psi=0,\quad  \psi'=0 \quad & \hbox{in} \quad ~  \Omega \end{array}
\right.
\end{equation} satisfies
\begin{equation}( \psi, \psi')\in C^0([0, T];  H^{1}(\Omega)\times  L^2(\Omega))
.\end{equation}

Next, we consider the following problem with inhomogeneous Neumann boundary conditions but without internal force terms:
\begin{equation}\left\lbrace
\begin{array}{ll}
\psi''-\Delta\psi +b(\psi)=0 \quad & \hbox{in} \quad (0,T) \times \Omega,\\
\partial_\nu \psi = e^{\lambda h}(e, DH)  & \hbox{on} \quad (0,T) \times\Gamma,\\
t=0:\quad  \psi=0,\quad  \psi'=0 & \hbox{in} ~ \quad   \Omega.\end{array}
\right.
\end{equation}
By \eqref{c31} and \eqref{c41} in Lemma \ref{Smoothneu}, we have
\begin{equation}
(\psi, \psi') \in C^0([0, T];  H^{\alpha}(\Omega)\times H^{\alpha-1}(\Omega))\end{equation}
and
\begin{equation}\psi|_{\Sigma}\in H^{2\alpha-1}(\Sigma)= H^{2\alpha-1}(0, T; L^2(\Gamma))\cap L^2(0, T; H^{2\alpha-1}(\Gamma)),
\end{equation}
where $\alpha$ is given by the first formula of \eqref{alfarobin}. Since this regularity result holds for all the eigenvectors of $B^T$, and all the eigenvectors of $B^T$ constitute  a set of basis in $\rr^N$,   we get the desired \eqref{sm1} and \eqref{sm2}.

We next consider the case that $\Omega$ is a parallelepiped. Although the boundary is only piecewise smooth, however, using direct eigenfunction expansions as for Theorem 6.1N.1 in \cite{Lasiecka2},  we can easily check that the regularity results (4.1)--(4.2)  remain true with $\alpha = 3/4-\epsilon$.
\end{proof}


\section{Exact boundary controllability and non-exact boundary controllability }
In this section, we will study  the exact boundary controllability and the non-exact boundary controllability for the coupled system \eqref{B} of wave equations with coupled $\text{Robin}$ boundary controls. We will prove that, for a  smooth  bounded domain  $\Omega\subset  \mathbb{R}^n$,   when the number of boundary controls is equal to $N$, the number of state variables,  system  \eqref{B} is exactly controllable for any given initial data $   (\widehat {U}_0, \widehat {U}_1)\in    ({\mathcal H}_1 )^N\times ({\mathcal H}_0)^N$, while, for a parallelepiped in $ \mathbb{R}^n$,  if $M=\text{rank}(D) <N$, namely, the number of controls needed is less than the number of variables, system \eqref{B} is not exactly  controllable in $({\cal H}_1 )^N\times ({\cal H}_0)^N$.

\subsection{Exact boundary controllability}
\begin{defn} System  \eqref{B}   is  exactly  null  controllable in the space $   (\mathcal  H_1)^N\times    (\mathcal  H_0)^N$,    if  there exists a positive constant $T>0$,  such that  for any given
$   (\widehat U_1, \widehat U_0)\in    (\mathcal  H_1)^N\times    (\mathcal  H_0)^N$, there exists a boundary control
 $H\in L^2   (0, T;    (L^2   (\Gamma))^M$,  such that  problem    \eqref{B}--\eqref{Bin}   admits a unique  weak solution $U$ satisfying the final condition
\begin{equation}t=T:\quad U=U'=0.\label {20.1}\end{equation}
 \end{defn}

\begin{thm}\label{controllable}
Assume that $M=\text{rank}(D)=N$. 
For a  smooth  bounded domain $\Omega\subset  \mathbb{R}^n  $,  system \eqref{B}  is exactly   controllable at a certain time $T>0$,  and the boundary control   continuously depends on the initial data:
\begin{align}\label{con1}
\|H\|_{L^2(0, T; (L^2(\Gamma))^N)}\leq c\| (\widehat{U}_0, \widehat{U}_1)\|_{({\cal H}_1 )^N\times ({\cal H}_0)^N},\end{align}
where  $c>0$ is a positive constant.
\end{thm}

\begin{proof}
We first consider the corresponding problem \eqref{Bn} and \eqref{Bin}. By Lemma \ref{controllableneu} and Remark \ref{1.1}, for any given initial data $(\widehat{U}_0, \widehat{U}_1)\in ({\cal H}_1 )^N\times ({\cal H}_0)^N$, there exists a boundary control $\widehat H \in  L_{loc}^2   (0, +\infty;    (L^2   (\Gamma))^N)$ with compact support in $[0,T]$, such that system \eqref{Bn}  with Neumann boundary controls is exactly   controllable at the time $T$,  and the boundary control  $\widehat{H}$  continuously depends on the initial data:
\begin{align}\label{bound}
\|\widehat{H}\|_{ L^2(0, T; (L^2(\Gamma))^N)}\leq c_1\| (\widehat{U}_0, \widehat{U}_1)\|_{({\cal H}_1 )^N\times ({\cal H}_0)^N},\end{align}
where $c_1>0$ is a positive constant.

Noting that $M=\text{rank}(D)=N$, $D$ is invertible and the boundary condition in system \eqref{Bn}\begin{equation}\label{newcontrol}\partial_\nu U  = D\widehat{H} \quad \text{on} ~(0,T)\times \Gamma \end{equation}
can be rewritten as
\begin{equation}\label{newconrobin}\partial_\nu U +BU=  D(\widehat{H}+D^{-1}BU) \overset{\text{def.}}{=}D H \quad \text{on} ~(0,T)\times \Gamma.\end{equation}
Thus, problem \eqref{Bn} and \eqref{Bin} with \eqref{newcontrol} can be equivalently regarded as problem \eqref{B}--\eqref{Bin} with \eqref{newconrobin}. In other words, the boundary control $H$ given by \begin{equation}\label{newconrobin5.5}
   H = \widehat{H}+D^{-1}BU   \quad \text{on} ~(0,T)\times \Gamma,\end{equation}
 where $U$ is the solution to problem  \eqref{Bn} and \eqref{Bin} with \eqref{newcontrol}, realizes  the exact boundary controllability of system \eqref{B}.

It remains to check that $H$ given by \eqref{newconrobin5.5} belongs to the control space $L^2   (0, T; (L^2 (\Gamma))^N)$ with  continuous dependence \eqref{con1}.
By the regularity result given in Theorem \ref{Smooth} (in which we take B=0), the trace $ U|_{\Sigma}\in (H^{2\alpha-1}(\Sigma))^N$, where $\alpha$ is defined by the first formula of  \eqref{alfarobin}. Since $2\alpha-1>0$,  we have $ H \in  L^2(0, T; (L^2(\Gamma))^N)$. Moreover, still by Theorem \ref{Smooth}, we have
\begin{align}\label{inner}
\|U\|_{ L^2(0, T; (L^2(\Gamma))^N)} \leq c_2 \big ( \|(\widehat{U}_0, \widehat{U}_1)\|_{({\cal H}_1 )^N\times ({\cal H}_0)^N}+\|\widehat{H}\|_{ L^2(0, T; (L^2(\Gamma))^N)} \big ),\end{align}where $c_2>0$ is another positive constant.
By the well-posedness theorem given in Theorem \ref{exist}, it is easy to see that system \eqref{B} is exactly controllable by the boundary control function $  H $. Moreover, noting \eqref{newconrobin5.5}, \eqref{con1} follows from \eqref{bound} and \eqref{inner}.  The proof is complete.
\end{proof}

\begin{rem}  \label{conner} The parallelepiped domain is only piecewise smooth, however,  the angles   between the corners  in a parallelepiped are all  equal to $\pi/2$, then, by Grisvard's results in  \cite{Gris} (see also p.534 of  \cite{Ala}),
the Laplacian $\Delta $ with Neumann boundary condition  and $L^2(\Omega)$ data has the $H^2(\Omega)$-regularity for a parallelepiped $\Omega\subset \mathbb R^n$ with $n\leqslant 3$. Therefore,  the  exact controllability Theorem 5.1 is still valid at least in this case.
\end{rem}

\subsection{Non-exact  boundary controllability}
Differently from   the case with Neumann boundary controls, the non-exact boundary controllability for the coupled system with coupled Robin boundary controls in a general domain  is still an open problem. Fortunately,  for some special domains, the  solution  to problem \eqref{B}--\eqref{Bin} may possess higher regularity. In particular, when $\Omega$ is  a parallelepiped, the optimal regularity of trace $U|_{\Sigma}$ almost reaches  $ (H^{\f{1}{2}}(\Sigma))^N$. This benefits a lot in the proof of the non-exact boundary controllability for the system with fewer boundary controls. We first give the following result of compactness, then we use it to prove the main Theorem \ref{uncontrollable} in this section that for  a parallelepiped $\Omega\subset\mathbb{R}^n$, 
if $M=\text{rank}(D) <N$, then system \eqref{B} is not exactly null controllable.


\begin{lemma} \label{Rnonbis}
Suppose that $\Omega \subset\mathbb{R}^n $ is a  parallelepiped.  Let ${\cal L}$ be a compact linear mapping from $L^2(\Omega)$ to $L^2(0, T; L^2(\Omega))$, and let
${\cal R}$ be a compact linear mapping from $L^2(\Omega)$ to $L^2(0, T; H^{1-\alpha}(\Gamma))$, where $\alpha$ is defined by \eqref{alfarobin}. Then, 
for any given $T>0$, there exists $\theta\in L^2(\Omega)$, such that  the solution to the following  
problem:
\begin{equation}\left\lbrace
\begin{array}{lll}
w''-\Delta w ={\cal L}\theta\qquad \quad & \hbox{in} \quad (0,T) \times\Omega,\\
\partial_\nu w= {\cal R}\theta  \quad \qquad   & \hbox{on} \quad (0,T) \times\Gamma,\\
t=0:\quad  w=0,\quad w'=\theta  \quad \qquad  & \hbox{in} ~\quad   \Omega \end{array}
\right.\label{c5}
\end{equation}
doesn't satisfy the  final condition
\begin{equation}
w(T)=w'(T)=0 \label{c6}.\end{equation}
 \end{lemma}

\begin{proof}
For any given $\theta\in L^2(\Omega)$, by Lemma  \ref{Smoothneu3}, the following  problem
\begin{equation}\left\lbrace
\begin{array}{lll}
\phi''-\Delta \phi =0 \qquad \quad & \hbox{in}
\quad (0,T) \times\Omega,\\
\partial_\nu \phi= 0  \qquad \quad & \hbox{on} \quad (0,T) \times\Gamma,\\
t=0:\quad  \phi=\theta,\quad \phi'=0  \qquad \quad & \hbox{in} \quad   \Omega  \end{array}
\right.\label{c7}
\end{equation}
admits a unique solution   $\phi$.
By \eqref{2.57} and \eqref{2.58} in Lemma \ref{Smoothneu3}, 
we have
\begin{equation}\|\phi\|_{L^2(0, T; L^2(\Omega))} \leq c \|\theta\|_{L^2(\Omega)} \label{c8'}
\end{equation}
and 
\begin{equation} \|\phi\|_{L^2(0, T; H^{\alpha-1}(\Gamma))}\leq c\|\theta\|_{L^2(\Omega)},
\label{c8}\end{equation}
where $\alpha$ is given by \eqref{alfarobin}.

On the other hand, by Lemma \ref{Smoothneu} and Lemma \ref{Smoothneu2}, problem \eqref{c5} admits a unique solution $w$. Assume by contradiction  that \eqref{c6} holds.  Then taking the inner product with $\phi$ on both sides of \eqref{c5} and integrating by parts,   it is easy to get
\begin{equation}\|\theta\|^2_{L^2(\Omega)} = \int_0^T\int_\Omega {\cal L}\theta\phi dx + \int_0^T\int_{\Gamma}{\cal R}\theta\phi d\Gamma.\end{equation}
Noting \eqref{c8'}--\eqref{c8}, we then have
\begin{equation}\|\theta\|_{L^2(\Omega)} \leq c (\|{\cal L} \theta\|_{L^2(0, T; L^2(\Omega))} +\| {\cal R} \theta\|_{L^2(0, T; H^{1-\alpha}(\Gamma))})\end{equation}for all $\theta\in L^2(\Omega)$,
which contradicts the compactness of $ \cal L$ and $\cal R$. 
\end{proof}


\begin{thm} \label{uncontrollable}
Assume that $M=\text{rank}(D)<N$.  
Assume furthermore  that $\Omega \subset \rr^n$ is a parallelepiped. 
Then, no matter how large $T>0$ is, system \eqref{B} is not exactly null controllable in the space $  ({\cal H}_1 )^N \times ({\cal H}_0)^N$. 
\end{thm}
\begin{proof}
Assume that $M=\text{rank}(D)<N$. Then there exists an $e\in \rr^N$, such that $D^Te=0$. Take the  special initial data
 \begin{equation} t=0:\quad U=0,\quad U'=e\theta \label{speini}\end{equation}
for system \eqref{B}.  Assume by contradiction  that the system is exactly   controllable  at the time $T >0 $. Then for any given $\theta\in L^2(\Omega)$, there exists a boundary control   $H \in  L^2(0, T, (L^2(\Gamma))^M)$, such that the corresponding solution satisfies
\begin{equation}U(T) = U'(T)=0.\label{0final}\end{equation}
Let
\begin{equation}w=(e, U),\qquad {\cal L}\theta = -(e, AU),\qquad {\cal R}\theta = -(e, BU)|_{\Sigma}.\end{equation}
Noting that $D^Te=0$, we see that $w$ satisfies problem \eqref{c5} and the final condition \eqref{c6}.

By Theorem \ref{Rnonbis}, in order to prove Theorem \ref{uncontrollable}, it suffices to show that the linear mapping  $ {\cal L}$ is  compact  from $L^2(\Omega)$ into $L^2(0, T; L^2(\Omega))$, and ${\cal R}$ is compact from $L^2(\Omega)$ into $H^{1-\alpha}(\Sigma)$, where $\alpha$ is given by \eqref{alfarobin}.

Since  system   \eqref{B}  with special initial data \eqref{speini} is exactly  null controllable, the linear mapping $\theta\rightarrow H$  is  continuous from $L^2   (\Omega)$ into $L^2   (0, T;    (L^2   (\Gamma))^M)$. By Theorem \ref{Smooth}, the   mapping $   (\theta, H) \rightarrow    (U, U') $ is  continuous  from $L^2   (\Omega) \times L^2   (0, T;    (L^2   (\Gamma))^M)$ into $C^0   ([0,T];    (H^{\alpha}   (\Omega))^N) \cap C^1   ([0,T];    (H^{\alpha-1}   (\Omega))^N)$. Besides, by Lions' compact embedding theorem  (Theorem 5.1 in \cite{Lions5}, p68),   the  following embedding
 $$L^2   (0, T;    (H^\alpha     (\Omega))^N)\cap H^1   (0, T;    (H^{\alpha-1}    (\Omega) )^N)\}\subset L^2   (0, T;    (L^2    (\Omega))^N)$$ is compact, hence the
linear mapping $\mathcal L$ is compact  from $L^2   (\Omega)$ into $L^2   (0, T;  L^2   (\Omega))$.

On the other hand, by \eqref{sm2} in Theorem \ref{Smooth}, $H\rightarrow U|_{\Sigma}$ is  continuous from $L^2(0, T; (L^2(\Gamma))^M)$ into $(H^{2\alpha-1}(\Sigma))^N$, then,  ${\cal R}:~\theta\rightarrow  -(e, BU)|_{\Sigma}$ is a continuous mapping from $L^2(\Omega)$ into $H^{2\alpha-1}(\Sigma)$.
When $\Omega$ is a parallelepiped,  $\alpha = 3/4-\epsilon$, then $2\alpha-1 >1-\alpha$. Hence, by  Simon's compact embedding result (Corollary 5 in \cite{Simon}, p86), the following embedding
$$H^{2\alpha-1}(\Sigma) =
L^2(0, T; H^{2\alpha-1}(\Gamma))\cap H^{2\alpha-1}(0, T; L^2(\Gamma))\subset L^2(0, T; H^{1-\alpha}(\Gamma))$$ is compact, therefore the mapping ${\cal R}$ is   compact from $L^2(\Omega)$ into $L^2(0, T; H^{1-\alpha}(\Gamma))$. The proof is complete.
\end{proof}

\begin{rem} \label{conner2}
We obtain the non-exact boundary controllability for system \eqref{B} with coupled Robin boundary controls in a parallelepiped  $\Omega$ when there is a  lack  of boundary controls.
The main idea is to use the  compact perturbation theory which  has a higher requirement on the regularity of the solution.  The improved  regularity  \eqref{sm1}--\eqref{sm2}  with $\alpha= 3/4-\epsilon$  for a parallelepiped domain of $\mathbb R^n$ is a consequence  of  Lasiecka-Triggiani's  sharp estimation for Neumann problem in \cite{Lasiecka2}.

On the other hand, the parallelepiped domain is only piecewise smooth.  However, by Remark \ref{conner}, 
the well-posedness Theorem \ref{Smooth}  and  the  exact controllability Theorem \ref{controllable} are still valid  for a parallelepiped $\Omega\subset \mathbb R^n$ with $n\leqslant 3$. Nevertheless,   since Theorem \ref{uncontrollable} takes 
the assumption that the system is  exactly controllable,   so, it is valid for all  parallelepiped  $\Omega\subset \mathbb R^n$  without any restriction on the dimension $n$.

In what follows,   all the results on the synchronization will be  established in  a smooth bounded domain, while the parallelepiped  domain  will be only used to examine  the necessity of compatibility conditions  (see \S7 below).    Theorem \ref{uncontrollable}   on the  non-exact controllability can be regarded as a start in this direction.  How to generalize this result to the   general domain is still an open problem.
\end{rem}

\section{Exact boundary synchronization by $p$-groups}\label{robingroup}
Based on the results of the exact boundary controllability and the non-exact boundary controllability, we continue to study the
exact boundary synchronization by $p$-groups  for system \eqref{B} with coupled Robin boundary controls. Theorem \ref{Snon2} will show that in order to obtain the exact boundary  synchronization by $p$-groups, we need at least $(N-p)$ boundary controls.

Let $p\geqslant  1$ be an integer and
\begin{equation} 0=n_0< n_1<n_2<\cdots<n_p=N\end{equation}
be  integers such that  $n_r-n_{r-1}\geqslant 2$ for $1\leqslant r \leqslant p$. We
re-arrange the components of the state variable $U$ into $p$ groups:
\begin{align}
(u^{(1)}, \cdots, u^{(n_{ 1})}),\quad  (u^{(n_{ 1}+1)}, \cdots, u^{(n_{2})}), \cdots, (u^{(n_{p-1}+1)}, \cdots, u^{(n_{p})}).
\end{align}

\begin{defn}\label{pdef}
System \eqref{B} is exactly   synchronizable by $p$-groups at the time $T>0$ in the space $({\cal H}_1)^N \times ({\cal H}_0)^N$, if for any given  initial data  $(\widehat{U}_0, \widehat{U}_1)\in ({\cal H}_1)^N \times ({\cal H}_0)^N$, there exists a boundary control $H \in (L_{loc}^2(0, +\infty ;L^2(\Gamma)))^M$ with compact support in $[0,T]$,  such that the corresponding solution $U=U(t,x)$ to   problem \eqref{B}--\eqref{Bin} satisfies
\begin{align}\label{synppp}
t\geq T: \quad u^{(i)}   =u_r (t,x),  
\qquad n_{r-1}+1\leq i \leq n_r,~ 1\leq r\leq p,
\end{align} 
where, $u=( u_1,\cdots, u_p)^T$, being unknown a priori, is called the corresponding exactly synchronizable state by $p$-groups.
\end{defn}

\begin{rem}
In particular, when $p = 1$, system  \eqref{B} is exactly  synchronizable. The following  theorems in this section  also work  in that case. 
\end{rem}





For the  given division $0=n_0<n_1<n_2<\cdots<n_p=N $, let $S_{r}$ be an $   (n_r-n_{r-1}-1)\times    (n_r-n_{r-1})$ full row-rank matrix:\begin{align}\label{cs}
 S_{r}=\left(\begin{array}{ccccc}
1&-1&&&\\
&1&-1&&\\
&&\ddots&\ddots&\\
&&&1&-1\end{array}\right) , \qquad 1\leq r\leq p,
\end{align}
and let $C_p$ be the following $(N-p)\times N$ matrix of synchronization by $p$-groups:
\begin{align}\label{c}
C_p= \left(\begin{array}{cccc}S_{1 } & & &\\
&S_{ 2}& &\\ & & \ddots & \\ & &  &S_{ p}  \end{array}\right).
\end{align} Evidently, we have
\begin{align}\label{kerc1}
\hbox{Ker}(C_p)=\hbox{Span}\{e_1,\cdots, e_p\}, \end{align}
where for $ 1\leq r\leq p$,
$$(e_r)_i=\begin{cases}
1,\qquad n_{r-1}+1 \leq i \leq   n_r , \\
0,\qquad \text{others}.
\end{cases}$$
Thus, the exact boundary synchronization by $p$-groups \eqref{synppp} can be equivalently written as
\begin{align}\label{8.6}
t\geq T: \quad C_p U \equiv 0 
 \end{align}or
\begin{align}\label{f1}
t\geq T: \quad U= \sum_{r=1}^p u_re_r.
 \end{align}

\begin{thm}\label{psugroup}
Assume that  $\Omega\subset  \mathbb{R}^n$ is a  smooth  bounded domain. 
Let $C_p$ be the $(N-p)\times N$ matrix of synchronization by $p$-groups defined by \eqref{cs}--\eqref{c}.  Assume that  both $A$ and $B$ satisfy the following conditions of $C_p$-compatibility: \begin{equation}\label{cpcom}
A \text{Ker}(C_p)\subseteq \text{Ker}(C_p),\qquad B \text{Ker}(C_p)\subseteq \text{Ker}(C_p).\end{equation}
Then there exists a boundary control matrix $D$ satisfying
 \begin{equation}
M=\hbox{rank} (D)=\hbox{rank} (C_pD)
=N-p,\end{equation}
such that system \eqref{B} is exactly   synchronizable by $p$-groups, and the corresponding boundary control   $H$ possesses the following continuous dependence:
\begin{align}\label{con22}
\|H\|_{L^2(0, T, (L^2(\Gamma))^{N-p})}\leq c\| C_p(\widehat{U}_0, \widehat{U}_1)\|_{(\mathcal H_1)^{N-p} \times (\mathcal H_0)^{N-p}},\end{align}
where $c>0$ is  a positive constant.
\end{thm}

\begin{proof}
Since   both $A$ and $B$ satisfy the conditions of $C_p$-compatibility \eqref{cpcom}, by Lemma 3.3 in \cite{Wei},  there exist matrices $\overline  A_p$ and $ \overline B_p$ of order  $(N-p)$, such that
\begin{equation}\label{acp}C_pA =\overline A_pC_p,\qquad C_pB =\overline B_pC_p.\end{equation}
Applying $C_p$ to problem  \eqref{B}--\eqref{Bin}
and defining
\begin{equation}W=C_pU,\qquad \overline D_p=C_pD,\end{equation}we have
\begin{equation}\left\lbrace
\begin{array}{lll}W''-{\Delta} W+\overline A_pW=0\qquad  & \hbox{in} \quad (0,+\infty) \times \Omega,\\
\partial_\nu W + \overline B_pW = \overline D_pH \hskip1.0cm & \hbox{on} \quad  (0,+\infty) \times \Gamma
\end{array}
\right.\label{c99}
\end{equation}
with the initial data:
\begin{align}\label{c99ini}
t=0:\quad W=C_p\widehat{U}_0,\quad W'=C_p\widehat{U}_1  \quad  \hbox{in} \quad   \Omega.\end{align}
Since $C_p$ is a surjection from $\rr^N$ to $\rr^{N-p}$, 
the exact boundary synchronization by $p$-groups for system \eqref{B} is equivalent to  the exact boundary   controllability for the reduced system \eqref{c99}, and the boundary control $H$, which realizes the exact boundary   controllability for the reduced system \eqref{c99}, must be   the boundary control which realizes the exact boundary synchronization by $p$-groups for system \eqref{B}.

Let $D$  be defined by
\begin{equation}\label{8.18}
\text{Ker}(D^T)=\hbox{Span}\{e_1,\cdots, e_p\}=\text{Ker}(C_p).
\end{equation}
We have $M=\hbox{rank} (D)=N-p$, and
\begin{align}
\text{Ker}(C_p) \cap \text{Im}(D)=\text{Ker}(C_p) \cap \{\text{Ker}(C_p)\}^{\perp}=\{0\}.\end{align}
By  Lemma 2.2 in \cite{Rao8},  we get $\hbox{rank} (C_pD)=\hbox{rank} (D)=M=N-p$, thus $\overline D_p$ is an invertible matrix of order $(N-p)$. By Theorem \ref{controllable}, the reduced system \eqref{c99} is exactly null controllable, then system \eqref{B} is exactly synchronizable by $p$-groups. By \eqref{con1}, we get \eqref{con22}.
\end{proof}

\begin{rem}
Noting Theorem \ref{controllable}, it is easy to check from the above proof that as long as \eqref{6.8} holds, system \eqref{B} must be  exactly synchronizable by $p$-groups under the assumptions of Theorem \ref{psugroup}. Noticing \eqref{6.9}, in fact,  Theorem \ref{psugroup} gives a way to find the boundary control matrix $D$ with minimum rank, but it is not the unique way.
\end{rem}

The well-posedness of the reduced problem \eqref{c99}--\eqref{c99ini} is guaranteed by the following

\begin{lemma}\label{sym}
If $B$ is similar to a symmetric matrix and  satisfies the condition of $C_p$-compatibility, then the reduced matrix $\overline  B_p$ of $B$, given by \eqref{acp},
is also similar to a symmetric matrix.
\end{lemma}
\begin{proof}
Since  $B$ is similar to a symmetric matrix,   there exists  a symmetric matrix $\widehat{B}$ and an invertible matrix $P$  such that $B=P \widehat{B} P^{-1}$. By the second formula of \eqref{acp},
 we have
\begin{align} 
 C_pBPP^TC_p^T
=\overline B_pC_pPP^TC_p^T.
\end{align}
Hence we get 
\begin{align}\overline B_p=C_pP \widehat{B}  P^TC_p^T(C_pPP^TC_p^T)^{-1}, \end{align}
 which is similar to the symmetric matrix
\begin{align}(C_pPP^TC_p^T)^{-\f{1}{2}} C_pP \widehat{B} P^TC_p^T(C_pPP^TC_p^T)^{-\f{1}{2}}.\end{align}
The proof is complete.\end{proof}

\section{Necessity of  the conditions of $C_p$-compatibility}\label{robincompatible}

In this section, we will discuss the necessity of the conditions of $C_p$-compatibility. This  problem is  closely related to the number of applied boundary controls.  The consideration  will be based on Theorem \ref{uncontrollable}, therefore,  in this section $\Omega$ is  a parallelepiped.

\subsection{Condition of $C_p$-compatibility for the internal coupling matrix $A$}

\begin{thm} \label{Snon2}Let $\Omega \subset \rr^n$ be a parallelepiped. 
Assume that system \eqref{B} is  exactly synchronizable by $p$-groups. Then   we  have
\begin{equation}\label{6.8}\hbox{rank}(C_pD)= N-p.\end{equation}
In particular,  we have
\begin{equation}M= \hbox{rank}(D)\geq N-p.\label{6.9}\end{equation}
\end{thm}

\begin{proof}
 If $\text{Ker}(D^T)\cap \text{Im}(C_p^T)  =\{0\}$,  by Lemma 2.2 in \cite{Rao8}, we have \begin{equation}\hbox{rank}(C_pD)=\hbox{rank}(D^TC_p^T)=\hbox{rank}( C_p^T)=  N-p.\end{equation}
Next, we prove that it is impossible to have $\text{Ker}(D^T)\cap \text{Im}(C_p^T)\not =\{0\}$. Otherwise, there exists a a vector $E \not= 0$, such that
\begin{equation}D^TC_p^TE=0.\end{equation}

Since system \eqref{B} is  exactly synchronizable by $p$-groups, taking the  special initial data \eqref{speini} for any given $\theta\in L^2(\Omega)$, the solution $U$  to problem \eqref{B} and \eqref{speini}   satisfies \eqref{synppp} (or \eqref{8.6}) under  boundary control $H \in  L^2(0, T, (L^2(\Gamma))^M)$.
Let
\begin{equation}w=(E, C_pU),\qquad {\cal L}\theta =-(E, C_pAU),\qquad {\cal R}\theta =-(E, C_pBU).\end{equation}
We easily get
\begin{align}
w(T)
=w'(T)
= 0.
\end{align}
Thus, we get again problem \eqref{c5} for  $w$. Besides, the exact boundary synchronization by $p$-groups for system \eqref{B} indicates that the final condition \eqref{c6} holds.  We then get a contradiction to Lemma \ref{Rnonbis}.
\end{proof}

We  have the following theorem on the condition of $C_p$-compatibility for the coupling matrix $A$:
\begin{thm}\label{pne}
Let $\Omega \subset \rr^n$ be a parallelepiped.   Assume that $ M= \text{rank}(D)  = N-p$.  If system \eqref{B} is exactly synchronizable by $p$-groups, then the coupling matrix $A=(a_{ij})$ should  satisfy the following condition of $C_p$-compatibility
\begin{align}\label{comppp}
A\text{Ker}(C_p) \subseteq \text{Ker}(C_p).
\end{align}
\end{thm}
\begin{proof}
It suffices to prove that
\begin{align}\label{7.2222}
C_pAe_r =0, \quad 1 \leq r \leq p.\end{align}

By \eqref{f1}, taking  the inner product with $C_p$ on both sides of the equations in system \eqref{B}, we get
\begin{equation}
t\geq T:\qquad  \sum_{r=1}^pC_pAe_ru_r =0\qquad \hbox{in}\quad \Omega.\end{equation}
If \eqref{7.2222} fails,  then there exist constant coefficients $\alpha_r (1 \leq r \leq p)$, not all equal to zero, such that
\begin{equation}\sum_{r=1}^p\alpha_ru_r=0\qquad \hbox{in}\quad \Omega.\label{g1}\end{equation} 
Let 
\begin{equation}\label{9.4} c_{p+1} = \sum_{r=1}^p \f{\alpha_re_r^T}{\|e_r\|^2}.\end{equation}
Noting $(e_r, e_s)=\|e_r\|^2 \delta_{rs}$, we have
\begin{equation}\label{9.5}t\geq T:\qquad  c_{p+1} U = \sum_{r=1}^p\alpha_ru_r=0\qquad \hbox{in}\quad \Omega.\end{equation}
Let
\begin{equation}\widetilde C_{p-1} = \begin{pmatrix}C_p\\
 c_{p+1}\end{pmatrix}.\end{equation}
We can easily get  \begin{equation}\label{7.888} t\geq T:\qquad \widetilde C_{p-1}  U  =0\qquad \hbox{in}\quad \Omega.\end{equation}
Noting \eqref{kerc1} and \eqref{9.4}, it is easy to see that  $c_{p+1}^T \not \in \hbox{Im}(  C_{p}^T)$, then, rank$(\widetilde C_{p-1}) = N-p+1.$  Since $M= \text{rank}(D)  = N-p$, we have $\hbox{Ker}(D^T)\cap \hbox{Im}(\widetilde C_{p-1}^T)\not =\{0\}$, then there exists a vector $E \not = 0$, such that
\begin{equation}D^T\widetilde C_{p-1}^TE=0.\end{equation}

Since system \eqref{B} is  exactly synchronizable by $p$-groups, taking the  special initial data \eqref{speini} for any given $\theta\in L^2(\Omega)$,  the solution $U$  to problem \eqref{B} and \eqref{speini}   satisfies \eqref{synppp} (or \eqref{8.6}) under  boundary control $H \in  L^2(0, T, (L^2(\Gamma))^M)$.
Let
\begin{equation}w=(E, \widetilde C_{p-1}U),\qquad {\cal L}\theta =-(E, \widetilde C_{p-1}AU),\qquad {\cal R}\theta =-(E, \widetilde C_{p-1}BU).\end{equation}
We get again problem \eqref{c5} for $w$. Noting  \eqref{7.888},  we have
\begin{equation}t=T:~~w(T)=0.\end{equation}
Similarly, we have $w'(T)=0$, then \eqref{c6} holds. Noting that $\Omega \subset \rr^n$ is a parallelepiped, similarly to the proof of Theorem \ref{Snon2},  we get a conclusion that contradicts Lemma \ref{Rnonbis}.
\end{proof}

\begin{rem}
The condition of $C_p$-compatibility  \eqref{comppp} is equivalent to the fact that there exist constants $\alpha_{rs}~(1 \leq r,s \leq p)$ such that \begin{align} \label{f11111}
 Ae_r= \sum_{s=1}^p \alpha_{sr}e_s \quad  1 \leqslant  r \leqslant  p, \end{align}
 or $A$ satisfies the following row-sum condition by blocks:
\begin{align}\label{asf1}
 \sum_{j=n_{r-1}+1}^{n_{r}}a_{ij}= \alpha_{sr}, \quad n_{s-1}+1 \leq i \leq n_{s}, \quad  1 \leq  r,s \leq p. \end{align}

In particular, when
\begin{align}
\alpha_{sr}=0, \quad  1 \leq r,s \leq p, \end{align}
we say that $A$   satisfies the zero-sum condition by blocks.
In this case, we have
\begin{align} \label{zerosum}
Ae_r= 0, \quad 1 \leq r \leq  p. \end{align}
\end{rem}



\subsection{Condition of $C_p$-compatibility for the boundary coupling matrix $B$}

Comparing with the internal coupling matrix $A$, the study on the necessity of the  condition of $C_p$-compatibility for  the boundary coupling matrix $B$  is more complicated. It concerns  the regularity of   solution to the problem with coupled  Robin boundary conditions.

Let
\begin{align}
\varepsilon_i=(0,\cdots, \overset{(i)}{1},\cdots, 0)^T, \qquad   1 \leq i \leq N \end{align}
be a set of classical orthogonal basis in $\rr^N$, and let
\begin{align}\label{9.15} V_r=\text{Span}\{\varepsilon_{n_{r-1}+1}, \cdots, \varepsilon_{n_{r}}\},  \qquad   1 \leq r \leq p.
\end{align}
Obviously, we have
\begin{align}\label{24.10} e_r\in V_r,  \quad   1 \leqslant  r \leqslant  p.
\end{align}

In what follows, we will discuss the necessity of the condition of $C_p$-compatibility for the boundary coupling matrix $B$ under the assumption that $Ae_r \in V_r$ and $Be_r \in V_r~   (1 \leqslant  r \leqslant  p)$.

\begin{thm}  \label{compacity/B} Let $\Omega  \subset \rr^n$ be a parallelepiped.
 Assume that $ M= \text{rank}(D)  = N-p$ and
\begin{equation}\label{invacon} Ae_r \in V_r,\quad Be_r \in V_r \quad    (1 \leqslant  r \leqslant  p).\end{equation} If system \eqref{B} is exactly synchronizable by $p$-groups, then the boundary coupling matrix $B$  should  satisfy the following condition of $C_p$-compatibility:
\begin{equation}B\text{Ker}(C_p)\subseteq \text{Ker}(C_p).\label{bcompati}\end{equation}\end{thm}

\begin{proof}
By \eqref{f1},   we have
\begin{equation}\label{psyn}\left\lbrace
\begin{array}{lll}\displaystyle \sum_{r=1}^p    (u''_re_r- \Delta u_re_r +  u_rAe_r)=0  & \hbox{in }     (T,+\infty) \times \Omega,\\
\displaystyle \sum_{r=1}^p    (\partial_\nu u_re_r  + u_rBe_r) =0   & \hbox{on }    (T,+\infty) \times \Gamma.
\end{array}
\right.
\end{equation}
Noting \eqref{24.10}--\eqref{invacon} and the fact that subspaces $V_r    (1 \leqslant  r \leqslant  p)$ are orthogonal to each other, for $1 \leqslant  r \leqslant  p$ we have
\begin{equation}\label{sep} \left\lbrace
\begin{array}{lll} u''_re_r- \Delta u_re_r +  u_rAe_r=0  & \hbox{in }      (T,+\infty) \times \Omega,\\\noalign{\medskip}
 \partial_\nu u_re_r  + u_r Be_r=0   & \hbox{on }     (T,+\infty) \times \Gamma.
\end{array}
\right.
\end{equation}
Taking the inner product with $C_p$ on both sides of the boundary condition on $\Gamma$ in \eqref{sep}, and noting \eqref{kerc1}, we get
\begin{equation} u_r C_p Be_r \equiv 0 \quad   \hbox{on }    (T,+\infty) \times \Gamma,\quad 1 \leqslant  r \leqslant  p.
\end{equation}
We claim that $C_p Be_r=0~(r=1,\cdots, p)$, which just mean that $B$ satisfies the   condition of $C_p$-compatibility \eqref{bcompati}.  Otherwise,  there exists an $ \bar r~   (1 \leqslant   \bar r   \leqslant  p)$  such that   $C_p Be_{ \bar r}\not =0$, consequently, we have
\begin{equation}u_{ \bar r} \equiv 0 \quad   \hbox{on }    (T,+\infty) \times \Gamma.\end{equation}
Then, it follows from  the boundary condition in system  \eqref{sep} that  \begin{equation}\partial_\nu u_{ \bar r} \equiv 0 \quad   \hbox{on }    (T,+\infty) \times\Gamma.\end{equation}
Hence, applying  Holmgren's uniqueness theorem (Theorem 8.2 in \cite{Lions3}) to  \eqref{sep}, we get
 \begin{equation}\label{l2}u_{ \bar r}\equiv 0 \quad  \hbox{in }    (T,+\infty) \times \Omega, \end{equation}
then it is easy to check that
\begin{equation}t \geqslant  T: \quad e_{\bar r}^T U
\equiv 0 \quad  \hbox{in }~ \Omega. \end{equation}
Let
\begin{equation} \widetilde C_{p-1} = \begin{pmatrix}C_p\\
 e_{\bar r}^T\end{pmatrix}.\end{equation}
We have  \eqref{7.888}.
Since $ e_{ \bar r}^T \not \in \hbox{Im}   ( C_p^T)$,  it is easy to show that rank$   (\widetilde C_{p-1}) = N-p+1$. On the other hand, since  $ \text{rank}   (\text{Ker}   (D^T))=p$, we have $\hbox{Ker}   (D^T)\cap \hbox{Im}   (\widetilde C_{p-1}^T)\not =\{0\}$, then,  there exists a vector  $E \not = 0$, such that
\begin{equation}D^T\widetilde C_{p-1}^TE=0.\end{equation}

Again, since system \eqref{B} is  exactly synchronizable by $p$-groups,  taking the  special initial data \eqref{speini} for any given $\theta\in L^2(\Omega)$, the solution $U$  to problem \eqref{B} and \eqref{speini}   satisfies \eqref{synppp} (or \eqref{8.6}) under  boundary control $H \in  L^2(0, T, (L^2(\Gamma))^M)$.
Let
\begin{equation}w= (E, \widetilde C_{p-1}U),\quad {\mathcal L}\theta =-(E, \widetilde C_{p-1}AU),\quad {\mathcal R}\theta =-(E, \widetilde C_{p-1}BU).\end{equation}
We have $ w(T)= w'(T)   \equiv 0,$
thus, we get again problem \eqref{c5} for $w$, and \eqref{c6} holds. Therefore, noting that $\Omega  $  is a parallelepiped, similarly to the proof of Theorem \ref{Snon2},  we  get a conclusion that contradicts Lemma \ref{Rnonbis}.
\end{proof}


\begin{rem}
Noting that  condition \eqref{invacon} obviously holds for $p=1$,  so, if system  \eqref{B} is exactly  synchronizable ($p=1$), the conditions of $C_1$-compatibility
 \begin{equation}
A\text{Ker}(C_1)\subseteq \text{Ker} (C_1) \quad \hbox{and}  \quad   B\text{Ker}(C_1)\subseteq \text{Ker} (C_1)\end{equation}
are always satisfied for both $A$ and $B$. 
\end{rem}

\subsection{Conditions of $C_2$-compatibility}

In this section, we will study the necessity of the condition of $C_2$-compatibility (when $p=2$) for $B$ for a specific example, where the restricted condition on $B$ given in \eqref{invacon} can be removed.

\begin{thm} \label{syn2group} Let $\Omega  \subset \rr^n$ be a parallelepiped.  
Assume that the coupling matrix  $A$ satisfies the zero-sum condition by blocks \eqref{zerosum}. Assume furthermore that system \eqref{B} is exactly  synchronizable by 2-groups with $M=\text{rank}(D)=N-2$. Then the coupling matrix $B$   necessarily satisfies the condition of $C_2$-compatibility:
\begin{equation}B\text{Ker}(C_2)\subseteq \text{Ker}(C_2).\label{comp2}\end{equation}\end{thm}

\begin{proof}
By the exact  boundary synchronization by 2-groups of \eqref{B}, we have
\begin{equation}\label{synpresent} t\geq T:\quad U= e_1u_1  +  e_2u_2\quad \hbox{in  } \Omega. \end{equation}
Noting \eqref{zerosum},   as $t\geq T$  we have
$$AU= Ae_1u_1+ Ae_2u_2=0,$$
then it is easy to see that
\begin{equation}\label{ss00}
\left\lbrace \begin{array}{lll}   U''- \Delta U =0 \qquad & \hbox{in}\quad (T, + \infty)\times \Omega,\\
\partial_\nu U + B U=0  & \hbox{on}\quad  (T, + \infty)\times  \Gamma.\end{array}
\right.
\end{equation}
Let $P$ be a matrix such that  $\widehat B= PBP^{-1}$ is a real symmetric matrix.  Denote
\begin{equation} u =    (u_1, u_2)^T.\end{equation}
Taking the inner product on both sides of \eqref{ss00} with $P^TPe_i$ for $i=1, 2$, we get
\begin{equation}\label{ss}
\left\lbrace
\begin{array}{lll}  L u''- L \Delta u =0 \qquad & \hbox{in}\quad (T, + \infty)\times \Omega,\\
L \partial_\nu u + \Lambda u=0  & \hbox{on}\quad  (T, + \infty)\times \Gamma,\end{array}
\right.
\end{equation}
where  the matrices $L$ and $\Lambda$ are given by
\begin{equation}L=    (Pe_i, Pe_j)  \quad \hbox{and}  \quad   \Lambda   =    (\widehat BPe_i, Pe_j),\quad 1\leqslant i, j\leqslant 2,\end{equation}
respectively. Clearly,  $L$ is a symmetric and  positive definite matrix and $\Lambda$ is a symmetric matrix.

Taking the inner product on both sides of  \eqref{ss} with $L^{-\frac{1}{2}}$ and denoting  $w=L^{\frac{1}{2}}u$, we get
\begin{equation}\label{ss2}
 \left\lbrace
\begin{array}{lll}  w''- \Delta w =0 \qquad & \hbox{in}\quad (T, + \infty)\times \Omega,\\
  \partial_\nu w + \widehat {\Lambda} w=0  & \hbox{on}\quad  (T, + \infty)\times \Gamma,\end{array}
\right.\end{equation}
where $\widehat {\Lambda}=L^{-\frac{1}{2}} \Lambda L^{-\frac{1}{2}}$ is also a symmetric matrix.

On the other hand, taking the inner product with $C_2$ on both sides of the boundary condition on $\Gamma$ in system  \eqref{B} and noting \eqref{synpresent}, we get
\begin{equation}\label{9.35rr}t\geq T:\qquad  C_2Be_1u_1 + C_2Be_2u_2 \equiv0\qquad\hbox{on}\quad \Gamma.\end{equation}

We  claim that $C_2Be_1= C_2Be_2=0$, namely, $B$  satisfies the condition of $C_2$-compatibility \eqref{comp2}.

Otherwise,   without loss of generality, we  way assume that  $C_2Be_1\not = 0$.
Then it follows from \eqref{9.35rr} that there  exists a non-zero vector $D_2\in \rr^2$, such that
\begin{equation}\label{24.25}t\geqslant  T:\quad  D_2^Tu\equiv 0 \quad\hbox{on } \Gamma. \end{equation}
Denoting
$$\widehat {D}^T= D_2^T L^{-\frac{1}{2}}, $$ we then have
\begin{equation}t\geq T:\qquad  \widehat {D}^Tw=0 \qquad\hbox{on}\quad \Gamma,\label{g3}\end{equation}
in which $w=L^{\frac{1}{2}}u$.
 By the multiplier method, we can prove that the following Kalman's criterion
\begin{equation}\hbox{rank}   ( \widehat D, \widehat {\Lambda}   \widehat D) =2\label{24.27}\end{equation}
is sufficient for the unique continuation  of  system \eqref{ss2}  under the observation \eqref{g3} on the infinite horizon $[T, +\infty)$ (see Theorem 3.22 and Remark 3.12 in \cite{Luthese}).  Since $M=\text{rank}(D)=N-2$, by Theorem \ref{uncontrollable}, system  \eqref{B}   is not exactly null controllable. So, the rank  condition \eqref{24.27}  does not hold. Thus, there exists a vector $E \not = 0$ in $\rr^2$, such that
 \begin{equation}\widehat {\Lambda}^T E= \widehat { \Lambda}  E=\mu E \quad \hbox{and}  \quad   \widehat D^TE=0.\label{24.28}\end{equation}
Noting \eqref{g3} and the second formula of \eqref{24.28}, we have both $E$ and $ w|_{\Gamma}\in \text{Ker}   ( \widehat D)$. Since $\text{dim }\text{Ker}   ( \widehat D)=1$,
there exists a constant $\alpha$ such that   $w=\alpha E$  on $\Gamma$. Therefore, noting the first formula of \eqref{24.28}, we have
 \begin{equation} \widehat {\Lambda} w=\widehat {\Lambda} \alpha E= \mu \alpha E=\mu w  \quad \text{on } \Gamma. \end{equation}
 Thus, \eqref{ss2} can be rewritten as
\begin{equation}\label{24.42}
\left\lbrace
\begin{array}{lll}  w''-\Delta w=0 \qquad &\hbox{in}\quad (T, + \infty)\times \Omega,\\
\partial_\nu w + \mu w=0  & \hbox {on}\quad (T, + \infty)\times \Gamma.\end{array}
\right.\end{equation}
Let $z= \widehat {D}^Tw$. Noting \eqref{g3},it follows from  \eqref{24.42} that
\begin{equation}  \left\lbrace
\begin{array}{lll}  z''-\Delta z =0  & \hbox{in  }     (T, + \infty)\times \Omega,\\
\partial_\nu z =z=0  & \hbox{on }     (T, + \infty)\times \Gamma.\end{array}\right.\end{equation}
Then, by Holmgren's uniqueness theorem, we have
\begin{equation} \label{24.29} t \geqslant  T:\quad z= \widehat D^Tw =D_2^Tu\equiv 0\quad \hbox{in }  \Omega.  \end{equation}

Let $D_2^T=   (\alpha_1, \alpha_2)$.  Define the following  row vector
\begin{equation}c_{3} = \f{\alpha_1 e_1^T}{\|e_1\|^2} + \f{\alpha_2 e_2^T}{\|e_2\|^2}.\end{equation}
Noting $   (e_1,e_2)=0$ and \eqref{24.29}, we have
\begin{equation}t\geqslant  T:\quad  c_{3} U = \alpha_1 u_1 +  \alpha_2 u_2=D_2^Tu\equiv 0 \quad \hbox{in }  \Omega.\end{equation}
Let
\begin{equation}\widetilde C_1 = \begin{pmatrix}C_2\\
 c_3\end{pmatrix}.\end{equation}
We get \begin{equation}\label{7.88882} t\geq T:\qquad \widetilde C_1  U  =0\qquad \hbox{in}\quad \Omega.\end{equation}
Since  $c_{3}^T \not \in \hbox{Im}   (  C_2^T)$, it is easy to see that rank$   (\widetilde C_1) = N-1,$ and $\hbox{Ker}   (D^T)\cap \hbox{Im}   (\widetilde C_1^T)\not =\{0\}$, thus there exists a vector  $\widetilde E \not = 0$, such that
\begin{equation} D^T\widetilde C_1^T\widetilde E=0.\end{equation}
Since system \eqref{B} is  exactly synchronizable by $2$-groups,  taking the  special initial data \eqref{speini} for any given $\theta\in L^2(\Omega)$,  the solution $U$ to problem \eqref{B} and \eqref{speini}   satisfies \eqref{synppp} (or \eqref{8.6}) under  boundary control $H $. Let
\begin{equation}u=   (\widetilde E, \widetilde C_1U),\quad {\mathcal L}\theta =-(\widetilde E, \widetilde C_1AU),\quad {\mathcal R}\theta =-(\widetilde E, \widetilde C_1BU).\end{equation}
We get again $ w(T)= w'(T)   \equiv 0$, and  problem \eqref{c5} of $w$  satisfies \eqref{c6}. Noting that $\Omega $ is a parallelepiped, similarly to the proof of Theorem \ref{Snon2}, we get a contradiction to  Lemma \ref{Rnonbis}.
\end{proof}

\section{Determination of  the exactly synchronizable state by $p$-groups}\label{robindetergroup}

In  general,  exactly synchronizable states by $p$-groups   depend not only on  initial data, but also  on  applied boundary controls. However, when the coupling matrices $A$ and $B$ satisfy certain algebraic conditions, the exactly synchronizable state by $p$-groups can be independent of  applied boundary controls. In this section, 
we first discuss the case when the exactly synchronizable state  by $p$-groups is independent of applied boundary controls, then we present the estimate on each exactly synchronizable state  by $p$-groups in general situation.

\begin{thm} \label{state/g} Let $\Omega \subset \mathbb{R}^n  $ be a smooth bounded domain. 
Assume that both $A$ and $B$ satisfy the conditions of $C_p$-compatibility \eqref{cpcom}. 
Assume furthermore that $A^T$ and $B^T$ possess a common invariant subspace $V$,  biorthogonal to  Ker$(C_p)$ (see Definition 2.1 in \cite{Rao8}). Then there exists a boundary control matrix $D$ with $M=\hbox{rank} (D)=\hbox{rank} (C_pD)=N-p$, such that system \eqref{B} is exactly   synchronizable by $p$-groups, and the exactly synchronizable state by $p$-groups $u=(u_1,\cdots, u_p)^T$ is independent of   applied boundary controls.
\end{thm}

\begin{proof}
Define the boundary control matrix $D$ by
 \begin{equation} \label{10.1}\hbox{Ker}(D^T)=V.\end{equation}
Since $V$ is biorthogonal to Ker$(C_p)$, by Lemma 2.5 in \cite{Rao8}, we have
\begin{equation}\label{10.2}\hbox{Ker}(C_p)\cap  \hbox{Im}(D) =\hbox{Ker}(C_p)\cap V^\perp =\{0\},\end{equation}
then, by Lemma 2.2 in \cite{Rao8},  we have
\begin{equation}\hbox{rank}(C_pD) =  \hbox{rank}(D)=M =N-p.\end{equation}
Therefore,  by Theorem \ref{psugroup}, system \eqref{B} is exactly   synchronizable by $p$-groups. Let $U$ be  the solution to problem \eqref{B}--\eqref{Bin}, which realizes the exact boundary synchronization  by $p$-groups at time $T>0$ under such $D$ and boundary control $H$.

By \eqref{10.2}, noting $\hbox{Ker}(C_p) =\text{Span} \{e_1,\cdots, e_p\}$, we may write
\begin{equation} V=\text{Span}\{E_1, \cdots, E_p\} \quad \text{with}~(e_r, E_s)=\delta_{rs} (r,s=1, \cdots, p). \end{equation}
Since $V$ is a common invariant subspace of $A^T$ and $B^T$,  there exist constants $\alpha_{rs}$ and $\beta_{rs} (r,s=1, \cdots, p)$ such that
\begin{equation}A^TE_r = \sum_{s=1}^p\alpha_{rs}
E_s, \qquad B^TE_r = \sum_{s=1}^p\beta_{rs}E_s.\end{equation}
For $r=1, \cdots, p$, let
\begin{equation}\phi_r = (E_r, U).\end{equation}
By system \eqref{B} and noting \eqref{10.1}, for $r=1, \cdots, p$ we have
\begin{equation}\left\lbrace
\begin{array}{lll}\displaystyle \phi_r'' -\Delta\phi_r +\sum_{s=1}^p\alpha_{rs}\phi_s=0 \qquad  & \hbox{in}\quad  (0,+\infty) \times \Omega,\\
\displaystyle \partial_\nu\phi_r  +\sum_{s=1}^p\beta_{rs}\phi_s=0     & \hbox{on}\quad  (0,+\infty) \times  \Gamma,\\
 t=0:\quad \phi_r=(E_r,\widehat{U}_0),\quad  \phi_r'=(E_r, \widehat{U}_1)  \qquad   & \text{in}\quad  \Omega. \end{array}
\right.\label{de}
\end{equation}
On the other hand, for $r=1, \cdots, p$ we have
\begin{equation}t\geq T:\quad \phi_r = (E_r, U) =  \sum_{s=1}^p(E_r, e_s)u_s =\sum_{s=1}^p\delta_{rs}u_s= u_r.\end{equation}
Thus, the exactly synchronizable state by $p$-groups $u=(u_1,\cdots, u_p)^T$ is entirely determined by the solution to problem \eqref{de}, which is independent of applied boundary controls $H$.
\end{proof}

The following result gives the counterpart of Theorem \ref{state/g}.

\begin{thm} \label{antistate/g} Let $\Omega \subset \mathbb{R}^n$ be a smooth bounded domain (say, with $C^3$ boundary).  Assume that both  $A$ and $B$ satisfy the conditions of $C_p$-compatibility \eqref{cpcom}. 
Assume furthermore that system \eqref{B} is exactly   synchronizable by $p$-groups. If there exists a subspace $V=\text{Span}\{E_1,\cdots, E_p\} $ of dimension $p$, such that the projection functions
\begin{equation} \phi_r=(E_r, U),\qquad r=1,\cdots, p \end{equation}
are independent of applied  boundary controls $H$, where $U$ is the solution to problem \eqref{B}--\eqref{Bin}, which realizes the exact boundary synchronization  by $p$-groups at time $T$,
then $V$ is a common invariant subspace of $A^T$ and $B^T$, $V \subseteq \text{Ker}(D^T)$, and biorthogonal to  Ker$(C_p)$.
\end{thm}

\begin{proof}
 Let  $(\widehat{U}_0,\widehat{U}_1)=(0,0)$.  By Theorem \ref{Smooth}, the  linear mapping
$$F:\quad H \rightarrow    (U, U') $$
is continuous from $L^2   (0, T;  (L^2 (\Gamma))^M)$ to $ C^0   ([0, T];     (H^{\alpha}   (\Omega))^N \times (H^{\alpha-1}   (\Omega))^N)$, where  $\alpha$ is defined by   \eqref{alfarobin}.
Let $F'$ denote  the Fr\'echet derivative of  the application $F$.  For any given $\widehat H\in L^2   (0, T; (L^2 (\Gamma))^M)$,  we define
\begin{align}
\widehat {U}=F'   (0)\widehat H.
\end{align} By   linearity, $\widehat {U}$ satisfies a system similar  to that of $U$:
\begin{align} \label{25.6}\begin{cases} \widehat {U} ''-{\Delta} \widehat {U} +A\widehat {U} =0\quad &  \hbox{in }        (0,+\infty) \times \Omega,\\
\partial_\nu \widehat {U}  + B\widehat {U}  = D \widehat H& \hbox{on }     (0,+\infty) \times \Gamma\\
t=0: \quad  \widehat {U}= \widehat {U}'=0  & \hbox{in } \Omega.
\end{cases}
\end{align}
Since the projection functions $ \phi_r=   (E_r, U)~   (r=1,\cdots, p)$ are independent of  applied boundary controls $H$,
we have
\begin{align}\label{25.7}
   (E_r,\widehat {U}) \equiv 0, \quad  \forall~ \widehat H \in L^2   (0, T;   (L^2 (\Gamma))^M),\quad r=1,\cdots, p.
\end{align}

First, we prove that $E_r \not \in \hbox{Im}   (C_p^T)$ for $r=1,\cdots, p$. Otherwise, there exist an $\bar r$ and a vector  $R_{\bar r} \in\rr^{N-p}$, such that $E_{\bar r}=C_p^TR_{\bar r} $, then we have
\begin{align}\label{25.8}
0=    (E_{\bar r},\widehat {U})=     (R_{\bar r}, C_p\widehat {U}),  \quad  \forall \widehat H \in L^2   (0, T;   (L^2   (\Gamma))^M).
\end{align}
Since $C_p\widehat {U}$ is the solution to the corresponding reduced problem \eqref{c99}--\eqref{c99ini}, noting the equivalence between  the exact boundary synchronization  by $p$-groups for the original system and the exact boundary controllability  for the reduced system,  from the exact boundary synchronization by $p$-groups for system  \eqref{B}, we know that the reduced system \eqref{c99} is exactly  controllable, then the value of $C_p\widehat {U}$ at the time $T$ can be chosen arbitrarily, thus we get
$R_{\bar r}=0,$
 which contradicts $E_{\bar r}\not=0$. Then, we have $E_r \not\in  \hbox{Im}   (C_p^T)~   (r=1,\cdots, p) $.
Thus $V \cap \{\text{Ker}   (C_p)\}^{\perp}=V \cap \text{Im}   (C_p^T)=\{0\}$. Hence by Lemma 4.2 and Lemma 4.3 in \cite{Luthese}, $V$ is bi-orthonormal to Ker$   (C_p)$, and
then  $   (V, C_p^T
)$ constitutes   a set of basis in $\rr^N$. Therefore, there exist constant coefficients $\alpha_{rs}~   (r,s=1,\cdots, p)$ and vectors $P_r \in \rr^{N-p}~   (r=1,\cdots, p)$, such that
\begin{equation}\displaystyle A^TE_r =\sum_{s=1}^p \alpha_{rs} E_s + C_p^TP_r, \quad r=1,\cdots, p.\end{equation}
Taking the inner product with $E_r$ on both sides of the equations in \eqref{25.6} and noting \eqref{25.7},  we get
\begin{equation}
0=   (A\widehat {U},E_r) =   (\widehat {U},A^TE_r) =   (\widehat {U},C_p^TP_r)=   (C_p\widehat {U},P_r)\end{equation}
for  $r=1,\cdots, p.$ Similarly, by the exact boundary controllability for the reduced system \eqref{c99}, we get $P_r=0~   (r=1,\cdots, p)$, thus we have
$$\displaystyle  A^TE_r=  \sum_{s=1}^p \alpha_{rs} E_s, \quad r=1,\cdots, p, $$
which means  that $V$ is an invariant subspace of $A^T$.

On the other hand, noting \eqref{25.7} and  taking the inner product with $E_r$ on both sides of the boundary condition on $\Gamma$ in \eqref{25.6}, we get
\begin{equation}\label{25.9}
   (E_r, B\widehat {U})=   (E_r, D\widehat H) \quad  \hbox{on }  \Gamma, \quad r=1,\cdots, p.\end{equation}
By Theorem \ref{Smooth}, for $  r=1,\cdots, p$ we have
\begin{align}\label{25.10}
&\|   (E_r, D\widehat  H)\|_{H^{2\alpha-1}   (\Sigma)} \\ =&\|   (E_r, B\widehat {U})\|_{H^{2\alpha-1}   (\Sigma)} \leqslant  c\|\widehat  H\|_{L^2   (0, T;   (L^2   (\Gamma))^M)},\notag\end{align} where $\alpha$ is given by  \eqref{alfarobin}.

We claim that $D^TE_r=0$ for $r=1,\cdots, p$. Otherwise, for $r=1,\cdots, p$, setting $\widehat  H=D^TE_rv$, it follows from \eqref{25.10} that
\begin{equation}
 \|v\|_{H^{2\alpha-1}   (\Sigma)}  \leqslant  c\|v\|_{L^2   (0, T;L^2   (\Gamma))}.\end{equation}
Since $2\alpha-1>0$, it contradicts the compactness of $H^{2\alpha-1}   (\Sigma) \hookrightarrow L^2   (\Sigma)$.
Thus, by \eqref{25.9} we have
\begin{equation}\label{25.11}
   (E_r, B\widehat {U})=0 \quad \hbox {on }    (0,T) \times \Gamma, \quad r=1,\cdots, p.\end{equation}

Similarly, there exist constants $\beta_{rs}~   (r,s=1,\cdots, p)$ and vectors $Q_r \in \rr^{N-p}~   (r=1,\cdots, p)$, such that
\begin{equation}\displaystyle B^TE_r = \sum_{s=1}^p \beta_{rs}E_s + C_p^TQ_r, \quad r=1,\cdots, p.\end{equation}
Substituting it into \eqref{25.11} and noting \eqref{25.7}, we have
\begin{equation}\displaystyle    \sum_{s=1}^p \beta_{rs}(E_s,\widehat {U}) +   ( C_p^TQ_r,\widehat {U})=    ( Q_r,C_p\widehat {U})=0, \quad r=1,\cdots, p.\end{equation}
By the exact boundary controllability for the reduced system \eqref{c99} , we get $Q_r=0~   (r=1,\cdots, p)$, then we have
\begin{equation}B^TE_r= \sum_{s=1}^p \beta_{rs}E_s, \quad r=1,\cdots, p,\end{equation}
which indicates that $V$ is also an invariant subspace of  $B^T$. The proof is complete.
\end{proof}

\begin{rem}
 When $\Omega \subset \mathbb{R}^n$ is a  parallelepiped, Theorem \ref{antistate/g} is still valid with the same proof.
\end{rem}

When $A$ and $B$ do not satisfy all the conditions mentioned in Theorem \ref{state/g},  exactly synchronizable states by $p$-groups may depend on  applied boundary controls. We have the following

\begin{thm} Let $\Omega \subset \mathbb{R}^n$ be a smooth bounded domain.  Assume that both  $A$ and $B$ satisfy the conditions of $C_p$-compatibility \eqref{cpcom}. 
 Then there exists a boundary control matrix $D$  such that system \eqref{B} is exactly   synchronizable by $p$-groups, and each exactly synchronizable state by $p$-groups $u=(u_1,\cdots, u_p)^T$ satisfies the following estimate:
\begin{equation}\label{estimate2}\|(u, u')(T)-(\phi, \phi')(T)\|_{(H^{\alpha+1}(\Omega))^{p} \times (H^{\alpha}(\Omega))^{p}}\leq
c\|C_p(\widehat{U}_0, \widehat{U}_1)\|_{(\mathcal H_1)^{N-p} \times (\mathcal H_0)^{N-p}},\end{equation}
where $\alpha$ is defined by the first formula of  \eqref{alfarobin}, $c $ is a positive constant and $\phi =(\phi_1,\cdots, \phi_p)^T$  is the solution to the following problem $(1 \leq r \leq p)$:
\begin{equation}\left\lbrace
\begin{array}{lll}
\displaystyle \phi_r'' -\Delta \phi_r +  \sum_{s=1}^p\alpha_{rs} \phi_s=0\qquad &\hbox {in}\quad (0,+\infty) \times \Omega, \\
\displaystyle \partial_\nu\phi_r +\sum_{s=1}^p\beta_{rs} \phi_s=0\hskip1.5cm & \hbox{on}\quad (0,+\infty) \times \Gamma,\\
t=0:\qquad \phi_r= (E_r, \widehat{U}_0),\quad \phi_r'=(E_r, \widehat{U}_1)  \qquad &\hbox {in}\quad \Omega, \end{array}
\right.
\end{equation}  in which
\begin{equation} \label{E} Ae_r= \sum_{s=1}^p\alpha_{sr}e_s,\quad Be_r= \sum_{s=1}^p\beta_{sr}e_s, 
\quad  r=1,\cdots, p.\end{equation}
\end{thm}

\begin{proof}
We first show that there  exists   a subspace $V$,  which is invariant for $B^T$ and bi-orthonormal to Ker$   (C_p)$.

Let $ B=P^{-1}\Lambda P,$ where $P$ is an invertible matrix, and $\Lambda$ be a symmetric matrix. Let $V=\hbox{Span}\{E_1,\cdots, E_p\}$ in which
\begin{equation} E_r= P^{T}Pe_r,\qquad r=1,\cdots, p.\end{equation}
Noting \eqref{kerc1}  and the fact that $\hbox{Ker}(C_p)$ is an invariant subspace of $B$, we get \begin{equation}B^TE_r  =  P^{T}PBe_r  \subseteq P^{T}P\hbox{Ker}(C_p) \subseteq V\qquad r=1,\cdots, p,\end{equation} then
$V$ is invariant for $B^T$.

We next show that $V^\perp\cap \text{Ker}(C_p)= \{0\}$. Then, noting that dim$   (V)$ = dim $  \text{Ker}(C_p) = p$, by Lemma 4.2 and Lemma 4.3 in \cite{Luthese}, $ V$ is bi-orthonormal to $\text{Ker}  (C_p)$.
For this purpose,  let $a_1,\cdots, a_p$ be  coefficients such that
\begin{equation} \sum_{r=1}^pa_re_r\in V^\perp.\end{equation}
Then
\begin{equation}    ( \sum_{r=1}^pa_re_r, E_s) =    ( \sum_{r=1}^pa_rPe_r, Pe_s) =0, \quad s=1, \cdots, p .\end{equation}
It follows that
\begin{equation}    ( \sum_{r=1}^pa_rPe_r,  \sum_{s=1}^pa_sPe_s )=0,\end{equation}
then $a_1=\cdots= a_p =0$, namely,  $V^\perp\cap \text{Ker}(C_p)= \{0\}$.

Denoting
\begin{align} Be_r= \sum_{s=1}^p \beta_{sr}e_s,  \quad  r=1, \cdots, p,\end{align}
a direct calculation yields that
 \begin{align}\label{25.14}  B^TE_r= \sum_{s=1}^p \beta_{rs}E_s, \quad  r=1, \cdots, p. \end{align}

Define the boundary control matrix $D$ by
\begin{equation}\text{Ker}(D^T) = V.\label{25.15c}\end{equation}
Noting \eqref{kerc1}, we have
\begin{align}&\hbox{Ker}   (C_p)\cap\hbox{Im}   (D) \\=&\hbox{Ker}   (C_p)\cap\{\hbox{Ker}   (D^T)\}^\perp =  \hbox{Ker}(C_p)\cap V^\perp = \{0\},\notag\end{align}
then, by Lemma 2.2 in \cite{Rao8}, we have \begin{equation}\hbox{rank}   (C_pD)=\hbox{rank}(D)=M =N-p.\end{equation}
Therefore,  by Theorem \ref{psugroup}, system \eqref{B} is exactly   synchronizable by $p$-groups. Let $U$ be the solution to problem \eqref{B}--\eqref{Bin}, which realizes the exact boundary synchronization  by $p$-groups at time $T$ under such $D$ and boundary control $H$.

Denoting  $\psi_r =    (E_r, U)    (r=1, \cdots, p)$, we have
\begin{align}
    &  (E_r, AU) =     (A^TE_r, U)\\= &      ( \sum_{s=1}^p  \alpha_{rs} E_s + A^T E_r -\sum_{s=1}^p \alpha_{rs} E_s, U)  \notag \\
=& \sum_{s=1}^p \alpha_{rs}    (E_s, U)+    ( A^T E_r -  \sum_{s=1}^p \alpha_{rs} E_s, U)\notag \\
=& \sum_{s=1}^p \alpha_{rs}\psi_s+    ( A^T E_r -  \sum_{s=1}^p \alpha_{rs} E_s, U).\notag
 \end{align}
By the assumption that $V$ is bi-orthonormal to $\hbox{Ker}   (C_p)$, without loss of generality, we may assume that
\begin{align}   (E_r, e_s)=\delta_{rs} \quad    (r, s=1,\cdots, p). \label{25.15}\end{align}
Then,  for any given $k = 1,  \cdots, p$, by  the first formula of  \eqref{E},  we get
\begin{align}
&   ( A^T E_r -  \sum_{s=1}^p \alpha_{rs} E_s, e_k)=   (E_r, Ae_k)- \sum_{s =1}^p \alpha_{rs}   (E_s, e_k)\notag\\ =&\sum_{s=1}^p \alpha_{sk}   (E_r,e_s)-\alpha_{rk}
= \alpha_{rk}-\alpha_{rk}=0,\notag
 \end{align}
hence
\begin{equation}A^T E_r -  \sum_{s=1}^p \alpha_{rs} E_s \in \{\text{Ker}(C_p)\}^\bot =\text{Im}   (C_p^T), \quad r=1, \cdots, p.\label{25.15b}\end{equation}
Thus, there exist $R_r \in \mathbb{R}^{N-p}    (r=1, \cdots, p) $, such that
\begin{align}
A^T E_r -  \sum_{s=1}^p \alpha_{rs} E_s=C_p^TR_r, \quad r=1, \cdots, p.
 \end{align}

Taking the inner product on both sides of  problem \eqref{B}--\eqref{Bin}   with $E_r$, and noting \eqref{25.14}--\eqref{25.15c}, for $r=1, \cdots, p$ we have
\begin{equation}\left\lbrace
\begin{array}{lll}
\displaystyle \psi_r'' -\Delta \psi_r +  \sum_{s=1}^p\alpha_{rs} \psi_s=-(R_r, C_pU) & \hbox {in }     (0,+\infty) \times \Omega, \\
\displaystyle \partial_\nu\psi_r +\sum_{s=1}^p\beta_{rs} \psi_s=0 & \hbox{on}     (0,+\infty) \times \Gamma,\\
t=0: \quad  \psi_r=    (E_r, \widehat {U}_0),\  \psi_r'=   (E_r, \widehat {U}_1) & \hbox {in } \Omega.\end{array}
\right.
\end{equation}
Then, by the classic semigroups theory, we have
\begin{align}\label{25.16}&\|   (\psi, \psi')   (T)-(\phi, \phi')   (T)\|_{   (H^{\alpha+1}   (\Omega))^{p} \times    (H^{\alpha}   (\Omega))^{p}} \leqslant  c_1\| (R_r, C_pU)\|_{  L^2(0,T;H^{\alpha}   (\Omega) ) }\\\leqslant & c_2\|C_p   (\widehat {U}_0, \widehat {U}_1)\|_{   (H^1   (\Omega))^{N-p} \times    (L^2   (\Omega))^{N-p}},\notag\end{align}
where $c_i$ for $ i=1,2 $ are different positive constants,   $\alpha$ is given by the first formula of \eqref{alfarobin}, and the second inequality follows from  \eqref{con1} and Theorem \ref{sm1}  since $C_pU$ is the solution to the reduced problem \eqref{c99}--\eqref{c99ini}.

On the other hand, noting \eqref{25.15}, it is easy to see that\begin{equation}t\geqslant  T:\quad \psi_r =    (E_r, U) = \sum_{s=1}^p   (E_r, e_s)u_s= u_r, \quad r=1, \cdots, p.\end{equation}
Substituting it into \eqref{25.16}, we get \eqref{estimate2}.
\end{proof}


Acknowledgement: The authors would like to thank the reviewers for their valuable and helpful suggestions.

\end{document}